\documentclass[11pt,twoside, leqno]{article}

\usepackage{amsthm}
\usepackage{amssymb}
\usepackage{amsmath}
\usepackage{mathrsfs}
\usepackage{txfonts}
\usepackage{graphics}
\usepackage{epsfig}

\allowdisplaybreaks \def\rr{{\mathbb R}}
\def\rn{{{\rr}^n}}

\def\cn{{\mathbb N}}

\def\fz{\infty}
\def\az{\alpha}
\def\supp{{\mathop\mathrm{\,supp\,}}}
\def\dist{{\mathop\mathrm {\,dist\,}}}
\def\diam{{\mathop\mathrm {\,diam\,}}}

\def\ez{\epsilon}

\def\bz{\beta}
\def\ro{\rho}

\def\gz{{\gamma}}

\def\sz{\sigma}

\def\ls{\lesssim}

\def\r{\right}
\def\lf{\left}

\newtheorem{thm}{Theorem}[section]
\newtheorem{lem}{Lemma}[section]
\newtheorem{prop}{Proposition}[section]
\newtheorem{rem}{Remark}[section]

\newtheorem{defn}{Definition}[section]

\numberwithin{equation}{section}

\begin{document}
\arraycolsep=1pt
\author{Renjin Jiang and Aapo Kauranen} \arraycolsep=1pt
\title{\Large\bf Korn inequality on irregular  domains}
 \footnotetext{
\endgraf}

\date{ }
\maketitle
\begin{center}
\begin{minipage}{10cm}\small
{\noindent{\bf Abstract.} In this paper, we study the weighted Korn
inequality on some irregular domains, e.g., $s$-John domains and
domains satisfying quasi-hyperbolic boundary conditions. Examples regarding
sharpness of the Korn inequality on these domains are presented. Moreover,
we show that Korn inequalities imply  certain Poincar\'e inequality.
 }\end{minipage}
\end{center}
\vspace{0.2cm}

\section{Introduction}
\hskip\parindent
Let $\Omega$ be a bounded domain of $\rr^n$, $n\ge 2$. For each
vector $\mathbf{v}=(v_1,\cdots,v_n)\in W^{1,p}(\Omega)^n$, let
$D\mathbf{v}$ denotes its gradient matrix, and $\epsilon
(\mathbf{v})$ denotes the symmetric part of $D\mathbf{v}$., i.e.,
$\epsilon(\mathbf{v})=(\epsilon_{i,j}(\mathbf{v}))_{1\le i,j\le n}$
with $$\epsilon_{i,j}(\mathbf{v})=\frac 12(\frac{\partial
v_i}{\partial x_j}+\frac{\partial v_j}{\partial x_i}). $$ Korn's
(second) inequality states that, if $\Omega$ is sufficient regular
(e.g., Lipschitz), then there exists $C>0$ such that
$$\int_\Omega|D\mathbf{v}|^p\,dx\le C\lf\{\int_\Omega| \ez(\mathbf{v})|^p\,dx+\int_\Omega| \mathbf{v}|^p\,dx\r\}.\leqno(K_{p})$$
The Korn inequality $(K_p)$ is a fundamental tool in the theory of linear elasticity equations;
see \cite{adm06,adl06,dl10,dl76,f47,ho95,ko89,ti01} and the references therein.
Notice that Korn inequality $(K_p)$ fails for $p=1$ even on a cube; see the example from \cite{cfm05}.

On $\rr^2$ and $p=2$, several different inequalities (including the Friedrichs' inequality)
are actually equivalent to Korn's inequality $(K_p)$ on simply connected Lipschitz domains; see \cite{hp83,ti01} for example.

Friedrichs \cite{f47} proved the Korn inequality $(K_p)$ for $p=2$ on
domains with a finite number of corners or edges on $\partial \Omega$,
 Nitsche \cite{nit81} proved the Korn inequality $(K_p)$ for $p=2$ on Lipschitz domains,
while Ting \cite{ti72} proved $(K_p)$ for all $p\in (1,\fz)$ by using Calder\'on-Zygmund theory;
Kondratiev and Oleinik \cite{ko89} studied the Korn inequality $(K_2)$ on star-shaped domains.
Recently, Acosta, Dur\'an and Muschietti \cite{adm06} proved the Korn inequality $(K_p)$
holds for all $p\in (1,\fz)$ on John domains.

Weighted Korn inequality on irregular domains (in particular, H\"older domains) have received
considerable interest recently; see \cite{adm06,adl06,adf13,dl10,ko89} and references therein. Motivated by this, in this
paper, we study weighted Korn inequality on some irregular domains including $s$-John domains $(s\ge 1)$
and domains satisfying quasihyperbolic boundary conditions.

We use  the divergence equation as the main tool, which is intimately connected to the weighted Poincar\'e inequality;
for the recent progress see  \cite{adm06,bb03,dmrt10,jkk13}. Our approach is in particular motivated
by \cite{dmrt10}.
Let $\Omega$ be a bounded domain in $\rr^n$. From \cite[Theorem 4.1]{dmrt10} the validity of Poincar\'e inequality
$$\int_\Omega|u(x)-u_{\Omega}|^p\,dx\le C\int_\Omega|\nabla u(x)|^p \dist(x,\partial\Omega)^b\,dx,$$
implies certain regularity of solutions to the divergence equation $\mathrm{div}\,{\mathbf u}=f$. Then by using duality, one gets
the (weighted) Korn inequality; see \cite{dl10} for instance. We in
Section 2 will generalize the arguments to more general settings to obtain the (weighted) Korn inequality.

We in Section 3 will go to $s$-John domains and domains satisfying quasi-hyperbolic boundary conditions, respectively.
By using Poincar\'e inequalities on these domains, we deduce the (weighted) Korn inequalities on them.
Moreover, we will show the obtained (weighted) Korn inequalities are sharp by presenting some counter-examples in Section 4.

Notice that the weighted Poincar\'e inequality on $s$-John domains is well known (see \cite{hk98,km00}), however,
there are no similar results on domains satisfying quasi-hyperbolic boundary conditions. To this end, we
will in Section 3 establish the weighted Poincar\'e inequality on such domains,
which may have independent interest.

Another interesting  question is what is the geometric counterpart of the Korn inequality.
In general the Korn inequality $(K_p)$ does not imply any Poincar\'e inequality.
Indeed, if $\Omega_1,\Omega_2\subset\rn$, $\Omega_1\cap \Omega_2=\emptyset$,
are two domains that support the Korn inequality $(K_p)$,
then $\Omega:=\Omega_1\cup \Omega_2$ admits the Korn inequality $(K_p)$ also. However, Poincar\'e inequality
does not have this property.

We in subsection 2.2 will show that, if the following Korn inequality
$$\int_\Omega|D\mathbf{v}|^p\,dx\le C\lf\{\int_\Omega| \ez(\mathbf{v})|^p\,dx+\int_Q| \mathbf{v}|^p\,dx\r\}\leqno(\widetilde{K}_{p})$$
holds for some cube $Q\subset\subset \Omega$, then there is a Poincar\'e inequality on $\Omega$.

The paper is organized as follows. In Section 2, we will show that, abstractly,
weighted Poincar\'e inequality implies a weighted Korn inequality; conversely, Korn inequality $(\widetilde{K}_{p})$ also
implies a Poincar\'e inequality.
In Section 3, we establish the Korn inequality on $s$-John domains and domains satisfying quasihyperbolic
boundary conditions, and present examples for the sharpness of the Korn inequality in Section 4.

Throughout the paper, we denote by $C$ positive constants which
are independent of the main parameters, but which may vary from
line to line. For $p\in [1,n)$, denote its Sobolev conjugate
$\frac{np}{n-p}$ by $p^\ast$.  Corresponding to to a function space $X$, we denote its $n$-vector valued
spaces by $X^n$. We will usually omit the superscript $n$ or $n\times n$ for simplicity.

\section{Korn inequality and Poincar\'e inequality}
\hskip\parindent
In this section, we  show that, abstractly, Poincar\'e inequality
implies Korn inequality; and  conversely, certain Korn inequality implies a Poincar\'e inequality.

Throughout the paper, let $\ro(x)$ be  the distance from $x$ to the boundary
$\partial \Omega$, i.e.,  $\ro(x):=\dist(x,\partial \Omega)$.
Let $a,b\in\rr$ and $p\in [1,\fz)$, the weighted Lebesgue space $L^p(\Omega,\ro^a)$ is defined as
set of all measurable functions $f$ in $\Omega$ such that
$$\|f\|_{L^p(\Omega,\ro^a)}:=\lf(\int_\Omega |f(x)|^p\ro(x)^a\,dx\r)^{1/p}<\fz.$$
We denote by $L^p_0(\Omega,\ro^a)$ the set of functions $f\in L^p(\Omega,\ro^a)$
with $\int_\Omega f(x)\ro(x)^a\,dx=0$.

The weighted Sobolev space $W^{1,p}(\Omega,\ro^a,\ro^b)$  is defined as
$$W^{1,p}(\Omega,\ro^a,\ro^b):=\lf\{u\in L^p(\Omega,\ro^a): \,\nabla u\in \mathscr{D}'(\Omega)\cap L^p(\Omega,\ro^b)\r\}$$
with the norm
$$\|u\|_{W^{1,p}(\Omega,\ro^a,\ro^b)}:=\|u\|_{L^p(\Omega,\ro^a)}+\|\nabla u\|_{L^p(\Omega,\ro^b)}.$$
We denote $W^{1,p}(\Omega,\ro^a,\ro^a)$ by $W^{1,p}(\Omega,\ro^a),$ and denote
$W^{1,p}(\Omega,\ro^a)$ by $W^{1,p}(\Omega)$ if $a=0$.

Notice that as $\ro^a$ and $\ro^b$ are continous positive functions in $\Omega$, smooth functions
$C^\fz(\Omega)\cap W^{1,p}(\Omega,\ro^a,\ro^b)$
is dense in $W^{1,p}(\Omega,\ro^a,\ro^b)$; see \cite[Theorem 3]{hk98}.

Let $p\ge 1$ and $a\ge 0$. We say that the $(P_{p,a,b})$-Poincar\'e inequality holds, if
there exists $C>0$ such that for every $u\in W^{1,p}(\Omega,\ro^a,\ro^b)$, it holds
$$\int_\Omega|u(x)-u_{\Omega,a}|^p\ro(x)^a\,dx\le C\int_\Omega|\nabla u(x)|^p\ro(x)^b\,dx,\leqno(P_{p,a,b})$$
where we denote by $u_{\Omega,a}:=\frac{1}{\int_{\Omega}\ro^a\,dx}\int_\Omega u\ro^a\,dx$ and
$u_{\Omega}:=u_{\Omega,a}$ for $a=0$.

\subsection{Korn inequality from Poincar\'e inequality}
\hskip\parindent In this subsection we will prove that Poincar\'e inequality implies Korn
inequality and in the following Section 3 we will provide examples which show
sharpness of our results.

\begin{thm}\label{t2.1}
Let $\Omega$ be a bounded domain of $\rr^n$, $n\ge 2$. Let
$p>1$, $a\ge 0$ and $b\in\rr$. Suppose the $(P_{p,a,b})$-Poincar\'e inequality holds on $\Omega$.
Then for an arbitrarily fixed  cube $Q\subset\subset \Omega$, there exists
$C=C(p,a,b,\Omega,Q)$ such that for every $\mathbf{v}\in W^{1,p}(\Omega,\ro^a)^n$, the following
inequality holds
$$\int_\Omega|D\mathbf{v}(x)|^p\ro(x)^a\,dx\le C\lf\{\int_\Omega| \ez(\mathbf{v})(x)|^p\ro(x)^{b-p}\,dx+\int_Q| D \mathbf{v}(x)|^p\ro(x)^{a}\,dx\r\}.\leqno{(\widetilde{K}_{p,a,b-p})}$$
\end{thm}

\begin{rem}\label{r2.1}\rm
If $a=0$ and $b=p$, then ($\widetilde{K}_{p,a,b-p}$) implies ($K_p$); see Kondratiev and Oleinik
\cite{ko89}. Indeed, as $Q\subset\subset \Omega$, it holds $\dist(Q,\partial \Omega)\le \ro(x)\le \diam(\Omega)$
for each $x\in Q$.
Since the Korn inequality $(K_p)$ holds on cubes, we always have
\begin{eqnarray*}
\int_Q| D \mathbf{v}|^p\ro^{a}\,dx&&\le C(a,\Omega,Q)\int_Q| D \mathbf{v}|^p\,dx\le C(a,\Omega,Q)\lf\{\int_Q| \ez(\mathbf{v})|^p\,dx+\int_Q| \mathbf{v}|^p\,dx\r\}\\
&&\le C(p,a,b,\Omega,Q)\lf\{\int_Q| \ez(\mathbf{v})|^p\ro^{b-p}\,dx+\int_Q| \mathbf{v}|^p\ro^{a}\,dx\r\}.
\end{eqnarray*}
Thus $(\widetilde{K}_{p,a,b-p})$ above implies that
$$\lf\|D\mathbf{v}\r\|_{L^p(\Omega,\ro^a)}\le C(p,a,b,\Omega,Q)
\lf\{\|\epsilon(\mathbf{v})\|_{L^p(\Omega, \ro^{b-p})}+\lf\|\mathbf{v}\r\|_{L^p(Q,\ro^a)}\r\}$$
and hence
$$\lf\|D\mathbf{v}\r\|_{L^p(\Omega,\ro^a)}\le C(p,a,b,\Omega,Q)
\lf\{\|\epsilon(\mathbf{v})\|_{L^p(\Omega, \ro^{b-p})}+\lf\|\mathbf{v}\r\|_{L^p(\Omega,\ro^a)}\r\},\leqno{({K}_{p,a,b-p})}$$
which is the usual Korn inequality $(K_p)$ if $a=0$ and $b=p$.
\end{rem}

We employ the divergence equation to prove the previous theorem.

Let $p,q\in (1,\fz)$ satisfying $1/q+1/p=1$, and $\Omega$ be a bounded domain in $\rn$. A vector function
$\mathbf{u}$ is called a solution to the divergence equation
$\mathrm{div}\,\mathbf{u}=f$
for some $f\in L^p_0(\Omega,\ro^a)$, if for every $\phi\in W^{1,q}(\Omega,\ro^a,\ro^b)$ it holds that
$$\int_\Omega \mathbf{u}(x)\cdot \nabla \phi(x)\,dx=\int_\Omega f(x) \phi(x) \ro(x)^a\,dx.\leqno{({\mathrm{div}}_{p,a,b})}$$
Recall that $C^\fz(\Omega)\cap W^{1,q}(\Omega,\ro^a,\ro^b)$ is dense in $W^{1,q}(\Omega,\ro^a,\ro^b)$.

\begin{prop}\label{p2.1}
Let $\Omega$ be a bounded domain in $\rn$, $p,q\in (1,\fz)$ with $1/p+1/q=1$, $a\ge 0$ and $b\in\rr$.
Suppose that $\Omega$ supports a $(P_{p,a,b})$-Poincar\'e inequality, then for each
$f\in L^q_0(\Omega,\ro^a)$, there exists $\mathbf{u}\in W^{1,q}(\Omega, \ro^{-qb/p},\ro^{q-qb/p})^n$
such that
$$\mathrm{div}\mathbf{u}=f\quad \mbox{in}\quad \mathcal{D}'(\Omega,\ro^a)$$
and
$$\|D \mathbf{u}\|_{L^q(\Omega,\ro^{q-qb/p})}\le C\|f\|_{L^q(\Omega,\ro^a)},$$
where $C=C(n,p,a,b)>0$.
\end{prop}
\begin{proof} The case $a=0$ is obtained in \cite[Theorem 4.1]{dmrt10}; the proof of the case $a>0$
is essentially  same as the case $a=0$ in \cite{dmrt10}, we outline the proof here.

For $f\in  L^q_0(\Omega,\ro^a)$, by using the $(P_{p,a,b})$-Poincar\'e inequality,
similarly as  \cite[Proposition 3.2]{dmrt10}, we conclude that there exists
a solution $\mathbf{u}$ to the equation $\mathrm{div}\,\mathbf{u}=f$
in  $\mathcal{D}'(\Omega,\ro^a)$ such that
$$\|\mathbf{u}\|_{ L^q(\Omega,\ro^{-qb/p})}\le \|f\|_{L^q(\Omega,\ro^a)}.$$

Let $\{Q_j\}_j$ be a Whitney decomposition of $\Omega$.
Similar to \cite[Proposition 4.2]{dmrt10}, we obtain a decomposition of $f$ as
$$ f(x)\ro(x)^a = \sum_{j\in I}f_j(x),$$
where $\{f_j\}$ satisfies:

(i) $\supp f_j\subset 2Q_j$;

(ii) $\int_{2Q_j}f_j(x)\,dx=0$;

(iii) $\sum_{j} \int_{2Q_j}|f_j(x)|^q\ro(x)^{q-qb/p}\,dx\le C\int_\Omega |f(x)|^q\ro(x)^a\,dx$.
for some $C=C(\Omega, p,a,b)$.

For each $j$, by \cite[Theorem 2]{bb03}, there exists $\mathbf{u}_j\in
W^{1,q}_0(2Q_j)^n$
such that $\mathrm{div}\mathbf{u}_j = f_j$ and
$$\|D\mathbf{u}_j\|_{L^q(2Q_j)}\le C(q)\|f_j\|_{L^q(2Q_j)}.$$

Denote $\mathbf{u}:=\sum_{j}u_j$. Since the dilations of Whitney cubes have
bounded overlap, one easily see that $\mathrm{div}\mathbf{u}= f$ holds in
$\Omega$. Indeed, for each
$\phi\in C^\fz(\Omega)$,
$$\int_\Omega\mathrm{div}\mathbf{u}(x)\phi(x)\,dx=\sum_{j} \int_\Omega\mathbf{u}_j(x)\cdot \nabla\phi(x)\,dx=
\int_\Omega \sum_{j}f_j(x)\cdot \phi(x)\,dx=\int_\Omega f(x) \phi(x)\ro(x)^a\,dx.$$
Moreover,
by using the property of Whitney decomposition again, i.e., $\ro(x)\sim \ell(Q_j)$ for each $x\in 2Q_j$
and each $j$, we further deduce that
\begin{eqnarray*}
  \|D \mathbf{u}\|^q_{L^q(\Omega,\ro^{q-qb/p})}&&\le \sum_{j}\int_{2Q_j} |D \mathbf{u_j}(x)|^q\ro(x)^{q-qb/p}\,dx\\
  &&\le C\sum_{j}\ell(Q_j)^{q-qb/p}\int_{2Q_j} |D \mathbf{u_j}(x)|^q\,dx\\
  &&\le C\sum_{j}\ell(Q_j)^{q-qb/p}\int_{2Q_j} |f_j(x)|^q\,dx\\
  &&\le C\sum_{j}\int_{2Q_j} |f_j(x)|^q \ro(x)^{q-qb/p}\,dx\\
  &&\le C \int_{\Omega} |f(x)|^q \ro(x)^{a}\,dx,
\end{eqnarray*}
which completes the proof.
\end{proof}

\begin{proof}[Proof of Theorem \ref{t2.1}]
Recall that  $D\mathbf{v}=(\frac{\partial v_i}{\partial x_j})_{1\le i,j\le n}$, $1\le i,j\le n$,
and $\epsilon(\mathbf{v})=(\epsilon_{i,j}(\mathbf{v}))_{1\le i,j\le n}$ with
$$\epsilon_{i,j}=\frac 12(\frac{\partial v_i}{\partial x_j}+\frac{\partial v_j}{\partial x_i})$$
and the identity
$$\frac{\partial^2v_i}{\partial x_j\partial x_k}=\frac{\partial \epsilon_{i,k}}{\partial x_j}
+\frac{\partial \epsilon_{i,j}}{\partial x_k}-\frac{\partial \epsilon_{j,k}}{\partial x_k}.$$
By this and using the properties of solutions to the divergence equations (Proposition \ref{p2.1}), we see that
for each $f\in L_0^q(\Omega,\ro^a)$ following holds
\begin{eqnarray*}
 \lf|\int_\Omega f(x)\ro(x)^a\lf(\frac{\partial v_j}{\partial
x_i}(x)-\lf(\frac{\partial v_j}{\partial x_i}\r)_\Omega\r)\,dx\r|&&=
\lf|\int_\Omega \mathrm{div}\,\mathbf{u}(x) \lf(\frac{\partial v_j}{\partial
x_i}(x)-\lf(\frac{\partial v_j}{\partial x_i}\r)_\Omega\r)\,dx\r|\\
 &&= \lf|\int_\Omega \mathbf{u}(x)\cdot \nabla \frac{\partial v_j}{\partial x_i}(x)\,dx\r|\\
 &&\le \|D\mathbf{u}\|_{L^q(\Omega, \ro^{q-qb/p})}\|\epsilon(\mathbf{v})\|_{L^p(\Omega, \ro^{b-p})}\\
 &&\le C\|f\|_{L^q(\Omega,\ro^a)}\|\epsilon(\mathbf{v})\|_{L^p(\Omega, \ro^{b-p})},
\end{eqnarray*}
which implies that
\begin{eqnarray}\label{2.1}
 \lf\|\frac{\partial v_j}{\partial x_i}-\lf(\frac{\partial v_j}{\partial x_i}\r)_\Omega\r\|_{L^p(\Omega,\ro^a)}\le C\|\epsilon(\mathbf{v})\|_{L^p(\Omega, \ro^{b-p})},
\end{eqnarray}

Now for an arbitrarily fixed cube $Q\subset\subset \Omega$, we choose a $\psi\in C_0^\fz(Q)$
such that $\supp \psi \subset Q$, $\int_Q\psi\,dx=1$ and $|\nabla \psi|\le C/\ell(Q)^{n+1}$.
Write
\begin{eqnarray}\label{2.2}
\frac{\partial v_j}{\partial x_i}=\frac{\partial v_j}{\partial x_i}-\lf(\frac{\partial v_j}{\partial x_i}\r)_\Omega
+\int_Q \lf[\lf(\frac{\partial v_j}{\partial x_i}\r)_\Omega-\frac{\partial v_j}{\partial x_i}\r]\psi\,dx+
\int_Q \frac{\partial v_j}{\partial x_i}\psi\,dx.
\end{eqnarray}
Then from the H\"older inequality, we obtain
\begin{eqnarray*}
\lf| \int_Q \lf[\lf(\frac{\partial v_j}{\partial x_i}\r)_\Omega-\frac{\partial v_j}{\partial x_i}\r]\psi\,dx \r|\le
C(a,p,Q,\Omega)  \|\epsilon(\mathbf{v})\|_{L^p(\Omega, \ro^{b-p})}
\end{eqnarray*}
and
$$\lf|\int_Q \frac{\partial v_j}{\partial x_i}(x)\psi(x)\,dx\r|\le C(a,p,Q,\Omega)
\lf\|\frac{\partial v_j}{\partial x_i}\r\|_{L^p(Q,\ro^a)},$$

Combining \eqref{2.1}, \eqref{2.2} and the above estimates, we obtain that
\begin{eqnarray*}
\lf\|\frac{\partial v_j}{\partial x_i}\r\|_{L^p(\Omega,\ro^a)}\le C(p,a,b,\Omega,Q)
\lf\{\|\epsilon(\mathbf{v})\|_{L^p(\Omega, \ro^{b-p})}+\lf\|\frac{\partial v_j}{\partial x_i}\r\|_{L^p(Q,\ro^a)}\r\},
\end{eqnarray*}
%
which is
$$\lf\|D\mathbf{v}\r\|_{L^p(\Omega,\ro^a)}\le C(p,a,b,\Omega,Q)
\lf\{\|\epsilon(\mathbf{v})\|_{L^p(\Omega, \ro^{b-p})}+\lf\|D\mathbf{v}\r\|_{L^p(Q,\ro^a)}\r\}.
\leqno (\widetilde{K}_{p,a,b-p})$$
The proof is completed.
\end{proof}

\subsection{Korn inequality implies Poincar\'e inequality }
\hskip\parindent
From the previous subsection, we know that the Poincar\'e inequality implies
Korn inequality, and in this section we will prove a partial converse result.

\begin{thm}\label{t2.2}
Let $\Omega$ be a bounded domain of $\rr^n$, $n\ge 2$. Let
$p>1$ and $Q\subset \Omega$ be a closed cube.
Suppose that
for all $\mathbf{v}\in W^{1,p}(\Omega)^n$ it holds that
$$\lf\|D\mathbf{v}\r\|_{L^p(\Omega)}\le C\lf\{\|\epsilon(\mathbf{v})\|_{L^p(\Omega)}+\lf\|D\mathbf{v}\r\|_{L^p(Q)}\r\},
\leqno ({\widetilde{K}_{p}})$$
then there exists $C>0$ such that for all $u\in W^{1,p}(\Omega)$, it holds
$$\int_\Omega |u(x)-u_{\Omega}|^p\,dx \le C\int_\Omega |\nabla u(x)|^p\,dx.
\leqno({\widetilde{P}_{p}})$$
\end{thm}

Notice that the Poincar\'e inequality $(\widetilde {P_{p}})$  is weaker than $(P_p)$.

We will need the following characterization of weighted  Poincar\'e inequality from
Haj{\l}asz and Koskela \cite[Theorem 1]{hk98} (for non-weighted cases see Maz'ya
\cite{maz85}).
A subset $A\subset \Omega$ is admissible if $A$ is open and $\partial A\cap \Omega$ is a smooth
submanifold.
\begin{thm}[\cite{hk98}]\label{t2.3}
  Let $\Omega$ be a bounded domain in $\rr^n$, $n\ge 2$. Let
$p\ge 1$ and  $a,b\in\rr$. Then the following conditions are
equivalent.

(i) There exists a constant $C>0$ such that, for every $u\in
C^\fz(\Omega)$ it holds that
$$\int_\Omega |u(x)-u_{\Omega,a}|^p\ro(x)^a\,dx \le C\int_\Omega |\nabla u(x)|^p\ro(x)^{b}\,dx.$$

(ii) For an arbitrary cube $Q\subset\subset \Omega$, there exists a
constant $C=C(Q)>0$ such that
\begin{equation}\label{2.3}
\int_A \ro(x)^a\,dx\le C \inf_{u}\int_\Omega |\nabla
u(x)|^p\ro(x)^{b}\,dx
\end{equation}
for every admissible set $A\subset \Omega$ with $A\cap Q=\emptyset$.
Here the infimum is taken over the set of all $u\in C^\fz(\Omega)$
that satisfy $u|_A= 1$ and $u|_Q=0.$
\end{thm}

We next prove Theorem \ref{t2.2}.

\begin{proof}[Proof of Theorem \ref{t2.2}]
We only need to verify that the second condition of Theorem \ref{t2.3} holds.
Assume that (${\widetilde{K}_{p}}$) holds. Fix a $y=(y_1,y_2,\cdots,y_n)\in\Omega$.

Let $A\subset \Omega$ with $A\cap Q=\emptyset$ be an admissible set, and
$u\in C^\fz(\Omega)$ that satisfies $u|_A= 1$ and $u|_Q=0.$

For each $x=(x_1,\cdots,x_n)\in \Omega$, let $\mathbf{v}=(v_1,v_2,0,\cdots,0)$ with
$$\lf\{
\begin{array}{ccc}
v_1(x_1,\cdots,x_n)&=&(x_2-y_2)u(x_1,\cdots,x_n),\\
v_2(x_1,\cdots,x_n)&=&(y_1-x_1)u(x_1,\cdots,x_n),\\
\end{array}
\r.
$$
Then for each $x=(x_1,\cdots,x_n)\in A$,
\begin{eqnarray*} D\mathbf{v}(x)=\lf(
\begin{array}{ccccc}
0 &  1 & 0& \cdots & 0\\
-1 & 0 & 0&\cdots & 0\\
0 & 0 & 0&\cdots & 0\\
\cdots\\
0 & 0 & 0&\cdots & 0
\end{array}
\r),
\end{eqnarray*}
and for $x\in Q$, $D\mathbf{v}(x)=0$. These imply that
$$\lf\|D\mathbf{v}\r\|_{L^p(\Omega)}^p\ge \int_A \,dx$$
and
$\|D(\mathbf{v})\|_{L^p(Q)}=0$. On the other hand, for every $x=(x_1,x_2,\cdots,x_n)\in \Omega$, it holds
\begin{eqnarray*}  D\mathbf{v}(x)=\lf(
\begin{array}{ccccc}
(x_2-y_2)\frac{\partial u}{\partial x_1} &  \ u+(x_2-y_2) \frac{\partial u}{\partial x_2}& (x_2-y_2)\frac{\partial u}{\partial x_3}&\cdots &(x_2-y_2)\frac{\partial u}{\partial x_n}\\
-u+(y_1-x_1)\frac{\partial u}{\partial x_1} & (y_1-x_1)\frac{\partial u}{\partial x_2}& (y_1-x_1)\frac{\partial u}{\partial x_3}&\cdots &
(y_1-x_1) \frac{\partial u}{\partial x_n}\\
0 & 0& 0&\cdots &0\\
\cdots\\
0 & 0& 0&\cdots &0
\end{array}
\r),
\end{eqnarray*}
which implies that  
$$\|\epsilon(\mathbf{v})\|_{L^p(\Omega)}\le C\|\,|\nabla u| (\cdot-y)\|_{L^p(\Omega)}
\le C\diam(\Omega)\|\nabla u\|_{L^p(\Omega)}.$$
The Korn inequality (${\widetilde{K}_{p}}$) implies that
\begin{eqnarray*}
\int_A\,dx&&\le \lf\|D\mathbf{v}\r\|^p_{L^p(\Omega)}\le C
\lf\{\|\epsilon(\mathbf{v})\|_{L^p(\Omega)}+\lf\|D(\mathbf{v})\r\|_{L^p(Q)}\r\}^p\nonumber\\
&& \le C\int_\Omega |\nabla u(x)|^p\,dx,
\end{eqnarray*}
for every $u\in C^\fz(\Omega)$ that satisfies $u|_A= 1$ and $u|_Q=0.$
Then the Poincar\'e inequality ($\widetilde{P}_{p}$) holds by using Theorem \ref{t3.2}.
\end{proof}

\begin{rem}\rm
Similarly, if the following Korn inequality
 $$\lf\|D\mathbf{v}\r\|_{L^p(\Omega)}\le C(\Omega,Q)
\lf\{\|\epsilon(\mathbf{v})\|_{L^p(\Omega)}+\lf\|\mathbf{v}\r\|_{L^p(Q)}\r\},
$$
holds for some cube $Q\subset\subset \Omega$, then the $({\widetilde{P}_{p}})$-Poincar\'e inequality also holds.
\end{rem}

\begin{rem}\rm
Theorem \ref{t2.2} can be generalized to the weighted cases by similar proofs as:
if the Korn inequality
$$\int_\Omega|D\mathbf{v}(x)|^p\ro(x)^a\,dx\le C\lf\{\int_\Omega| \ez(\mathbf{v})(x)|^p\ro(x)^{b-p}\,dx+\int_Q| D \mathbf{v}(x)|^p\ro(x)^{a}\,dx\r\}.\leqno{(\widetilde{K}_{p,a,b-p})}$$
holds for some $Q\subset\subset \Omega$, then the weighted Poincar\'e inequality
$$\int_\Omega |u(x)-u_{\Omega,a}|^p\ro(x)^a\,dx \le C\int_\Omega |\nabla u(x)|^p\ro(x)^{b-p}\,dx
$$
holds.
\end{rem}

\section{Korn inequality on general domains}

\hskip\parindent In this section, we are going to study the Korn inequality on some irregular domains.

If $\Omega$ is an $\az$-H\"older domain for some $\az\in (0,1] $, it is then proved in \cite{adl06}
that there is a constant $C=C(n,p,\Omega,\az)>0$ such that
for every $\mathbf{v}\in W^{1,p}(\Omega,\ro^a)^n$, it holds
\begin{equation*}
\int_\Omega|D\mathbf{v}|^p\ro^a\,dx\le C\lf\{\int_\Omega
|\ez(\mathbf{v})|^p\ro^{b-p}\,dx+\int_\Omega| \mathbf{v}|^p\ro^{a}\,dx\r\},
\end{equation*}
 where $0\le a=b-\az p$. See \cite{adl06,adf13} for more on this aspect and the
counterexample for sharpness.

We next focus on two kinds of irregular domains: $s$-John domains and quasi-hyperbolic domains.

\subsection{$s$-John domains}
\hskip\parindent Let us first recall the definition of $s-$John domain.

\begin{defn}[$s$-John domain]
A bounded domain $\Omega\subset \rn$ with a distinguished point
$x_0\in\Omega$ called an $s$-John domain, $s\ge 1$, if there exists a constant $C>0$ such that
for all $x\in\Omega$, there is a curve $\gz: [0,l]\to\Omega$
parametrised by arclength such that $\gz(0)=x$, $\gz(l)= x_0$, and
$d(\gz(t),\rn\setminus\Omega)\ge Ct^s$.
\end{defn}

If $s=1$ then we say that $\Omega$ is a John domain for simplicity.
John domains were introduced by Martio and Sarvas \cite{ms78}, F. John \cite{j61} had earlier
considered a similar class of domains.

The following Poincar\'e inequality is a special case from \cite{hk98,km00}.
\begin{thm}\label{t3.1}
Suppose that $\Omega$ is an $s$-John domain, $s\ge 1$ and $a\ge 0$.
Then there is a constant $C=C(n,p,\Omega,a,b)>0$ such that
$$\lf(\int_\Omega|u-u_{\Omega,a}|^p\ro^a\,dx\r)^{1/p}\le
C\lf(\int_\Omega|\nabla u|^p\ro^b\,dx\r)^{1/p}$$
for each $u\in C^\fz(\Omega)$, where
$n+a\ge s(n+b-1)-p+1.$
\end{thm}

We have the corresponding Korn inequality on $s$-John domains.
\begin{thm}\label{t3.2}
Suppose that $\Omega$ is an $s$-John domain with $s\ge 1$, and $a\ge 0$.
Then there is a constant $C=C(n,p,\Omega,a,b)>0$ such that
for every $\mathbf{v}\in W^{1,p}(\Omega,\ro^a)^n$, it holds
$$\int_\Omega|D\mathbf{v}|^p\ro^a\,dx\le C\lf\{\int_\Omega
|\ez(\mathbf{v})|^p\ro^{b-p}\,dx+\int_\Omega| \mathbf{v}|^p\ro^{a}\,dx\r\},\leqno({K}_{p,a,b-p})$$
 where $a\ge 0$ and $n+a\ge s(n+b-1)-p+1.$

 Moreover, for $a\ge 0$ and $n+a<s(n+b-1)-p+1$, there exists a domain $\Omega$
which does not support the Korn inequality $(K_{p,a,b-p}).$
\end{thm}
\begin{proof} By using Theorem \ref{t2.1} and the Poincar\'e inequality (Theorem \ref{t3.1}),
we see that the Korn inequality $(K_{p,a,b-p})$ holds if $n+a\ge s(n+b-1)-p+1.$
The converse part follows from the Example 4.1(1) in Section 4.
\end{proof}

\begin{rem}\rm
For the case $s=1$, i.e., on the John domain, we can then take $a=0$ and $b=p$, and obtain
the usual Korn inequality. This gives an another proof of \cite[Theorem
4.2]{adm06}
\end{rem}

\subsection{Quasihyperbolic domains}
\hskip\parindent
Let $\Omega$ be a proper domain in $\rn$, $n\ge 2$. By quasihyperbolic metric we mean
that for all $x,y\in \Omega$,
\begin{equation*}
 k(x,y):= \inf_{\gamma} \int_{\gamma} \frac{1}{\mathrm{dist}(z, \partial
\Omega)}\,ds(z),
\end{equation*} where the infimum is taken over all curves $\gamma$ joining $x$ to
$y$ in $\Omega$.
 The quasihyperbolic metric arises naturally in the theory of conformal geometry and plays
an important role for example in the study of the boundary behavior of
quasiconformal maps.

 Our domain $\Omega$ is said to satisfy a $\beta$-quasihyperbolic boundary
condition (for short, $\bz$-QHBC), if for some fixed base point $x_0$ there
exists $C_0<\infty$ such that for every $x\in\Omega$
\begin{equation*}
 k(x,x_0) \leq \frac{1}{\beta} \log \frac{\mathrm{dist}(x_0, \partial
\Omega)}{\mathrm{dist}(x,
\partial \Omega)} + C_0.
\end{equation*} Changing the base point $x_0$ changes the constant $C_0$.

We first establish the following weighted Poincar\'e inequality
on these domains; for non-weighted cases see \cite{SS90,kot02,jkk12},
and recent paper \cite{hmv12} for $(q,p)$-Poincar\'e inequality with $q<p$.

Let $\mathcal{W}$ be a Whitney decomposition of $\Omega$. We may and do assume that the
basepoint $x_0$ is the center of some $Q\in \mathcal{W}.$  For each $Q\in
\mathcal{W}$, we choose a quasihyperbolic geodesic $\gamma$ joining $x_0$ to the
center of $Q$ and let $P(Q)$ denote the collection of all of Whitney cubes that intersect
$\gamma$. The shadow of the cube $Q\in\mathcal{W}$ is the set
$$
S(Q):=\bigcup_{\substack{Q_1\in \mathcal{W}\\  Q\in P(Q_1)}} Q_1.
$$

We have the following estimate for the shadow of a cube from \cite{jkk12}.

\begin{lem}[\cite{jkk12}]\label{l3.1} Let $\Omega$ satisfy the
$\beta$-quasihyperbolic boundary condition, for some
$\beta\leq1$.
 There exists a constant $C=C(n,C_0)$ such
that for all $Q\in \mathcal{W}$
\begin{equation*}
 \mathrm{diam}(S(Q))\leq C\,
\mathrm{dist}(x_0,\partial\Omega)^{\frac{1-\beta}{1+\beta}}\,
\mathrm{diam}(Q)^{ \frac{2\beta }{ 1+\beta
}}.
\end{equation*}
\end{lem}

\begin{thm}\label{qhbcpoincare}
 Let $\Omega\subset \rr^n$ be a proper subdomain satisfying a
$\beta$-quasihyperbolic boundary condition, for some $\beta\leq 1.$ Then
there is a constant $C=C(n,p,q,\bz,\Omega)>0$ such that
$$\lf(\int_\Omega|u-u_{\Omega,a}|^q\ro^a\,dx\r)^{1/p}\le
C\lf(\int_\Omega|\nabla u|^p\ro^b\,dx\r)^{1/p}$$
for each $u\in C^\fz(\Omega)$, where
$ 1\le p\le q<\fz$, $a\ge 0$,
$$\frac{a+n}q\frac{2\beta}{1+\beta}+\frac{p-n-b}{p}>0;$$
additionally, $q\le \frac{np}{n-p}$ if $p<n$.
\end{thm}

\begin{proof}
For $p = 1$, the same proof as \cite[Proof of theorem 7]{hk98} applies with
Lemma \ref{l3.1}
replacing the $s$-John condition there.

For $p>1$, the proof is similar to the proof of theorem 3.2 in \cite{kot02} with
small modifications from \cite{jkk12}.
We will verify condition (ii) of Theorem
\ref{t2.3}. Let $\mathcal{W}$ be a Whitney decomposition of $\Omega$.
Let $A$ be an admissible set and $Q_0$ some fixed cube. Let $u$ be a
smooth test function which equals 1 on $A$ and 0 on $Q_0.$  We split our set
$A$ to two parts
$$
A_g=\{x\in A : u_Q \leq \frac12 \text{ for some Whitney cube } Q \ni x  \}
$$
and $A_b=A\setminus A_g.$
 For all points $x\in A_g$ with $x\in Q\in \mathcal{W}$, from the properties of
the  Whitney decomposition,  we have $\ro(x)\sim \ell(Q)$, and hence
\begin{eqnarray*}
\frac12 \lf(\int_{Q\cap A} \ro(x)^a\,dx \r)^{\frac{p}{q}}&&\leq C\ell(Q)^{ap/q}\left(\int_Q |u(x)-u_Q|^q\,dx
\right)^{\frac{p}{q}}\\
&&\leq C \ell(Q)^{\frac{ap}q+1-n+\frac {pn}q-b} \int_Q |\nabla u(x)|^p \rho(x)^b\,dx\\
&&\le C \diam(\Omega)^{\frac{ap}q+1-n+\frac {pn}q-b} \int_Q |\nabla u(x)|^p \rho(x)^b\,dx,
\end{eqnarray*}
where $\frac{ap}q+1-n+\frac {pn}q-b\ge p\lf(\frac{a+n}q\frac{2\bz}{1+\bz}+\frac {p-n-b}p\r)\ge 0$.

Summing  over all such cubes $Q$, as $q\ge p$, we obtain
\begin{equation}\label{thegood}
\int_{\Omega} |\nabla u|^p \rho(x)^b \,dx\ge C^{-1}\lf(\int_{A_g} \ro(x)^a\,dx \r)^{\frac{p}{q}}.
\end{equation}

Next we estimate the integral over the bad set. For each $x\in A_b$, let
$P(Q(x))$ consist of the collection of all of the Whitney cubes which
intersect the quasihyperbolic geodesic joining $x_0$ to the center of $Q(x)$,
then a straightforward chaining argument shows that
$$C\sum_{Q\in P(Q(x))} \diam Q\fint_Q |\nabla u(y)| dy\ge 1.$$
Hence, by using the H\"older inequality, we have
\begin{align*}
& \int_{A_b} \rho(x)^a dx\\
&\leq C\int_{A_b} \rho(x)^a \sum_{Q\in P(Q(x))} (\diam
Q)^{1-\frac{n}{p}} \left(\int_Q |\nabla u(y)|^p dy\right)^{1/p} dx
\\&= C \sum_{Q\in \mathcal W} \int_{S(Q)\cap A_b} \rho(x)^adx (\diam
Q)^{1-\frac{n}{p}} \left(\int_Q |\nabla u(y)|^p dy\right)^{1/p}
\\&\leq C \sum_{Q\in \mathcal W} \int_{S(Q)\cap A_b} \rho(x)^adx (\diam
Q)^{1-\frac{n}{p}-\frac{b}{p}} \left(\int_Q |\nabla u(y)|^p \rho(y)^b dy\right)^{1/p}
\\&\overset{\text{H\"older}}{\leq}
C\left(\sum_{Q\in \mathcal W}  \left(\int_{S(Q)\cap A_b}
\rho(x)^adx\right)^{p'} (\diam Q)^{(1-\frac{n}{p}-\frac{b}{p})p'}
\right)^{\frac{1}{p'}}
\left(\sum_{Q\in \mathcal W} \int_Q |\nabla u(y)|^p \rho(y)^b dy
\right)^{\frac{1}{p}}
\\&\leq
C\left(\sum_{Q\in \mathcal W}  \left(\int_{S(Q)\cap A_b}
\rho(x)^adx\right)^{p'} |Q|^{(\frac 1n-\frac{1}{p}-\frac{b}{pn})p'}
\right)^{\frac{1}{p'}}
\left(\int_{\Omega} |\nabla u(y)|^p \rho(y)^b dy
\right)^{\frac{1}{p}}.
\end{align*}
This together with the following Lemma \ref{l3.2} gives that
\begin{equation*}
\lf(\int_{A_b} \rho(x)^a dx\r)^{1/q} \le C \left(\int_{\Omega} |\nabla u(y)|^p \rho(y)^b dy
\right)^{\frac{1}{p}}.
\end{equation*}
The proof is completed by combining the above estimate together with \eqref{thegood}.
\end{proof}

\begin{lem}\label{l3.2}
With the assumptions of Theorem \ref{qhbcpoincare} and $p>1$, we have
\begin{equation*}
\sum_{Q\in \mathcal W}  \left(\int_{S(Q)\cap A_b}
\rho(x)^adx\right)^{p'} |Q|^{(\frac 1n-\frac{1}{p}-\frac{b}{pn})p'} \leq C
\left(\int_{A_b}
\rho(x)^adx\right)^{\frac{p'}{q'}}.
\end{equation*}
\end{lem}
\begin{proof} Since $q\ge p>1$, we have $p'-1-\frac{p'}{q}\ge 0$ and hence
 \begin{align*}
\displaystyle
  \sum_{Q\in \mathcal W}&  \left(\int_{S(Q)\cap A_b}
\rho(x)^adx\right)^{p'} |Q|^{(\frac 1n-\frac{1}{p}-\frac{b}{pn})p'}
\\&\leq \left(\int_{A_b}
\rho(x)^adx\right)^{p'-1-\frac{p'}{q}}
\sum_{Q\in \mathcal W}  \left(\int_{S(Q)}
\rho(x)^adx\right)^{\frac{p'}{q}}  \int_{S(Q)\cap A_b}
\rho(x)^adx |Q|^{(\frac 1n-\frac{1}{p}-\frac{b}{pn})p'}\\&=
\left(\int_{A_b}
\rho(x)^adx\right)^{p'-1-\frac{p'}{q}}
\sum_{Q\in \mathcal W}  \left(\frac{\left(\int_{S(Q)}
\rho(x)^adx\right)^{\frac{1}{q}}}{|Q|^{{\frac{1}{p}+\frac{b}{pn}-\frac 1n}}}\right)^{p'}
 \int_{S(Q)\cap A_b}\rho(x)^adx \\&=
\left(\int_{A_b}
\rho(x)^adx\right)^{p'-1-\frac{p'}{q}}
\sum_{Q\in \mathcal W} \sum_{Q'\in \mathcal S(Q)}
\left(\frac{\left(\int_{S(Q)}
\rho(x)^adx\right)^{\frac{1}{q}}}{|Q|^{{\frac{1}{p}+\frac{b}{pn}-\frac 1n}}}\right)^{p'}
 \int_{Q'\cap A_b}\rho(x)^adx \\&=
\left(\int_{A_b}
\rho(x)^adx\right)^{p'-1-\frac{p'}{q}}
\sum_{Q'\in \mathcal W} \sum_{Q\in P(Q')}
\left(\frac{\left(\int_{S(Q)}
\rho(x)^adx\right)^{\frac{1}{q}}}{|Q|^{\frac{1}{p}+\frac{b}{pn}-\frac 1n}}\right)^{p'}
 \int_{Q'\cap A_b}\rho(x)^adx \\&\leq C
\left(\int_{A_b}
\rho(x)^adx\right)^{p'-1-\frac{p'}{q}+1}=C
\left(\int_{A_b}
\rho(x)^adx\right)^{\frac{p'}{q'}}.
 \end{align*}
 Above in estimating the last inequality, we use Lemma \ref{l3.1} to see that
$$
\int_{S(Q)} \rho(y)^ady=\sum_{Q'\in S(Q)}\int_{Q'} \rho(y)^a dy\leq C
|Q|^{(\frac{a}{n}+1)\frac{2\beta}{1+\beta}},
$$
and \cite[Lemma 2.6]{kot02} to obtain
\begin{align}
\label{apulasku}
 \sum_{Q\in P(Q')}\left(\frac{\left(\int_{S(Q)}
\rho(x)^adx\right)^{\frac{1}{q}}}{|Q|^{\frac{1}{p}+\frac{b}{pn}-\frac 1n}}\right)^{p'}
\leq C \sum_{Q\in P(Q')}
|Q|^{((\frac{a}{n}+1)\frac{2\beta}{(1+\beta)q}+{\frac 1n-\frac{1}{p}-\frac{b}{pn}})p'}\leq
C(a,b,p,q,\beta,\Omega, n),
\end{align}
as
$$
(\frac{a}{n}+1)\frac{2\beta}{(1+\beta)q}+\frac 1n-\frac{1}{p}-\frac{b}{np}>0.
$$
The proof is completed.
\end{proof}

\begin{rem}\rm If $a=b=0$, then the Poincar\'e inequality obtained above coincides with
\cite[Theorem 1]{jkk12}. One can modify \cite[Example 5.5]{kot02} to show that
the
Poincar\'e inequality  from Theorem \ref{qhbcpoincare} is sharp, in the sense that the inequality
$$\lf(\int_\Omega|u-u_{\Omega,a}|^q\ro^a\,dx\r)^{1/p}\le
C\lf(\int_\Omega|\nabla u|^p\ro^b\,dx\r)^{1/p}$$
does not holds if
$\frac{a+n}q\frac{2\beta}{1+\beta}+\frac{p-n-b}{p}<0$.
\end{rem}

We have the following Korn inequality for domain satisfying a
$\beta$-QHBC.

\begin{thm}\label{t3.4}
 Let $\Omega\subset \rr^n$ be a proper subdomain satisfying a
$\beta$-QHBC, for some $\beta\leq 1.$ Let $p>1$.
Then there is a constant $C=C(n,p,q,\bz,\Omega)>0$ such that
for every $\mathbf{v}\in W^{1,p}(\Omega,\ro^a)^n$, it holds
$$\int_\Omega|D\mathbf{v}|^p\ro^a\,dx\le C\lf\{\int_\Omega
|\ez(\mathbf{v})|^p\ro^{b-p}\,dx+\int_\Omega| \mathbf{v}|^p\ro^{a}\,dx\r\},\leqno (K_{p,a,b})$$
where $a\ge 0$, $b\in \rr$ satisfying $(a+n)\frac{2\beta}{1+\beta}>{n+b-p}.$

 Moreover, for $a\ge 0$ and $(a+n)\frac{2\beta}{1+\beta}<{n+b-p}$,
 the Korn inequality $(K_{p,a,b-p})$ fails on $\Omega$.
\end{thm}
\begin{proof} By using Theorem \ref{t2.1} and the Poincar\'e inequality (Theorem \ref{qhbcpoincare}) with $p=q$,
we see that the Korn inequality $(K_{p,a,b-p})$ holds if $(a+n)\frac{2\beta}{1+\beta}+{p-n-b}>0.$

The converse part follows from Example 4.1(2) in Section 4.
\end{proof}

\begin{rem}\rm
Notice that in the Poincar\'e inequality (Theorem \ref{qhbcpoincare}) and the Korn inequality
(Theorem \ref{t3.4}), there are no result for the borderline case $(a+n)\frac{2\beta}{1+\beta}+{p-n-b}=0$.
However, we believe the Poincar\'e inequality  and the Korn inequality is true at the borderline.
\end{rem}

\section{Examples}

\hskip\parindent
We next give examples to indicate the sharpness of Theorems \ref{t3.2} and \ref{t3.4} for $n=2$.
It is easy to check that the example works also for higher dimension.

{\bf Example 4.1.}
Let $\Omega$ be a domain of the union of sequences of rectangles
$$
\Omega=Q_0\cup C_{1}\cup Q_1\cup C_{2}\cup Q_2\cup C_{3}\cup\ldots.
$$
The rectangles are arranged as in Figure 1. This is possible if the sidelengths
converge to 0 fast enough.
The sidelength of $Q_0$ is one and that of square $Q_i$ is $r_i.$
The heigth of the rectangle $C_{i}$ is $r_i^{\tau}$ and width is
$r_i^{\sigma}$
for all $i\ge 1$, where
$\sigma,\tau\geq 1$ is a fixed real number. The domain is called ``A rooms-and-corridors domain".

\begin{figure}[ht]
\label{fig1}
\centerline{ \epsfig{file=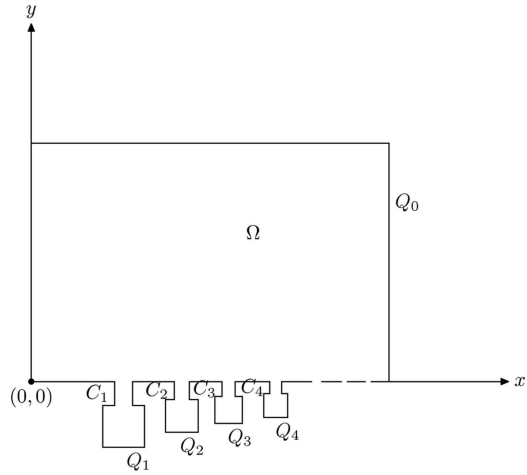, scale=0.9}
}
\caption{A rooms-and-corridors domain.}
\end{figure}

We can control the boundary accessibility by choosing the constants $\sz$ and
$\tau.$ Here are two relevant choices.
\begin{itemize}
\item[(i)] $\Omega$ is an $s$-John domain if $s=\sz$ and $\Omega$ is not an $s$-John domain if $s<\sz$
(independent of $\tau$); see \cite[Example 5.5]{kot02}.

\item[(ii)] $\Omega$ is a $\beta$-QHBC domain if $\sz\le \tau$, $\beta=\frac{1}{2\sigma-1}$ and ¦¸ is not a $\beta$-QHBC domain for
any $\beta >0$ if $1\le\tau<\sz$; see \cite[Example 5.5]{kot02} and
\cite{jkk12}.
\end{itemize}

For each $i\in \cn$, define the vector function $\mathbf{u}_i(x,y)$ on $\Omega$
as follows:
$$\mathbf{u}_i(x,y)=\lf\{
\begin{array}{cccc}
(2y+r_i^\tau, \ -2(x-x_i)),\ \ \ & \ \forall \ \ \ (x,y)\in Q_i\\
(-\frac{y^2}{r_i^\tau},\ \frac{2}{r_i^\tau}(x-x_i)y),\ \ & \ \forall \ (x,y)\in
C_i \\
(0,0),\ \hspace{2cm}   &\   \forall \ (x,y)\in \Omega\setminus (C_i\cup Q_i).
\end{array}
\r.
$$
Above $(x_i,-r_i/2-r_i^\tau)$ is the center of the cube $Q_i$. It is immediate
that $\mathbf{u}_i$
is Lipschitz continuous in $\Omega$.

Direct computation gives that when $(x,y)\in Q_i$,
\begin{eqnarray*} D\mathbf{u}_i(x,y)=\lf(
\begin{array}{cc}
0 &  2 \\
-2 & 0
\end{array}
\r),
\end{eqnarray*}
and hence $\epsilon(\mathbf{u}_i)(x,y)=0$; when $(x,y)\in C_i$,
\begin{eqnarray*} D\mathbf{u}_i(x,y)=\lf(
\begin{array}{cc}
0 &  -2y/r_i^\tau \\
2y/r_i^\tau & 2(x-x_i)/r_i^\tau
\end{array}
\r),
\end{eqnarray*}
and
\begin{eqnarray*} \epsilon(\mathbf{u}_i)(x,y)=\lf(
\begin{array}{cc}
0 &  0 \\
0 & 2(x-x_i)/r_i^\tau
\end{array}
\r).
\end{eqnarray*}
Meanwhile, for $(x,y)\in \Omega\setminus (C_i\cup Q_i)$,
$\epsilon(\mathbf{u}_i)(x,y)=D\mathbf{u}_i(x,y)=0$.

From the above calculations, we deduce that
\begin{equation*}
\int_\Omega|D\mathbf{u}|^p\ro^a\,dx\,dy\gtrsim \int_{Q_i}\ro^a\,dx\,dy\sim
r_i^{a+2},
\end{equation*}
and since  $\ro(x)=r_i^\sz-|x-x_i|$ in the corridor $C_i$, we see that
\begin{equation*}
\int_\Omega|\ez(\mathbf{v})|^p\ro(x)^{b-p}\,dx\,dy\sim
\int_{C_i}(r_i^\sz-|x-x_i|)^{b-p}\lf(\frac{|x-x_i|}{r_i^\tau}\r)^{p}\,dx\,dy\sim
r_i^{\sz (b+1)+\tau(1-p)}.
\end{equation*}

Moreover,
\begin{equation*}
\int_\Omega|\mathbf{u}|^p\ro^{a}\,dx\,dy\sim \int_{Q_i}r_i^p\ro^a\,dx\,dy\sim
r_i^{a+p+2}.
\end{equation*}
The above estimates imply that if the Korn inequality
\begin{equation*}
\int_\Omega|D\mathbf{v}|^p\ro^a\,dx\,dy\le C\lf\{\int_\Omega
|\ez(\mathbf{v})|^p\ro^{b-p}\,dx\,dy+\int_\Omega|
\mathbf{v}|^p\ro^{a}\,dx\,dy\r\}, \leqno(K_{p,a,b-p})
\end{equation*}
holds, then for each $i$, it holds
\begin{equation}
\label{contradiction}
r_i^{a+2}\ls r_i^{\sz (b+1)+\tau(1-p)}+r_i^{a+p+2}.
\end{equation}

By choosing different parameters $\sz$ and $\tau$, we obtain

\begin{itemize}

\item[(1)]{\bf Sharpness of Theorem \ref{t3.2}.} Let $1=\tau\le \sz$, from
Example 4.1 (i) we know $\Omega$ is a $s$-John domain and $s=\sz$.

In this case, if the Korn inequality $(K_{p,a,b-p})$ holds on $\Omega$, then \eqref{contradiction} becomes
$$
r_i^{a+2}\ls r_i^{\sz (b+1)+1-p}+r_i^{a+p+2}.
$$
This is true for all $i.$
Thus $a+2\geq \sz (b+1)+1-p$ and we see that Korn inequality $(K_{p,a,b-p})$ fails if
$a+2<s(b+1)+1-p,$ therefore our Theorem \ref{t3.2} is sharp.

\item[(2)]{\bf Sharpness of Theorem \ref{t3.4}.} Let $1\le \tau = \sz$, then
$\Omega$ is a $\beta$-QHBC domain with $\beta=\frac{1}{2\sigma-1}$
according to Example 4.1 (ii).

Suppose the Korn inequality $(K_{p,a,b-p})$ holds on $\Omega$, then \eqref{contradiction} becomes
$$
r_i^{a+2}\ls r_i^{\sz (b+2-p)}+r_i^{a+p+2}
$$
for each $i.$ This implies that $a+2 \geq \sz( b+2-p)$.
We see that Korn inequality $(K_{p,a,b-p})$ fails if
$ \frac {2\bz}{1+\bz}(a+2)=\frac 1\sz(a+2)<b+2-p,$ which implies that Theorem \ref{t3.4} is sharp.
\end{itemize}

\section*{Acknowledgment}
\hskip\parindent
The authors wish to thank their advisor Professor Pekka Koskela for posing the problem and
helpful discussions. Kauranen was supported by The
Finnish National Graduate School in Mathematics and its Applications.

\vspace{0.4cm}

\noindent Renjin Jiang$^{1}$ \& Aapo Kauranen$^{2}$

\

\noindent
1. School of Mathematical Sciences, Beijing Normal University,
Laboratory of Mathematics and Complex Systems, Beijing 100875, People's Republic of China

\

\noindent 2. Department of Mathematics and Statistics, University of Jyv\"{a}skyl\"{a}, P.O. Box 35 (MaD),
FI-40014
Finland

\

\noindent{\it E-mail addresses}:
\texttt{rejiang@bnu.edu.cn}

\hspace{2.3cm}
\texttt{aapo.p.kauranen@jyu.fi}
\end{document}